\documentclass[10pt]{amsart} 

\usepackage[psamsfonts]{amssymb} 
\usepackage{amsfonts,amsmath} 
\usepackage{graphicx, color}
\usepackage{enumerate}
\usepackage{url}
\usepackage{hyperref}
\usepackage[T1]{fontenc}
\usepackage[utf8]{inputenc}

\usepackage{algorithm}
\usepackage{algpseudocode}

\usepackage[all]{xy}

\newtheorem{thmA}{Theorem}

\newtheorem{theorem}{Theorem}[section]

\newtheorem{proposition}[theorem]{Proposition}
\newtheorem{lemma}[theorem]{Lemma}

\newtheorem{question}[theorem]{Question}
\newtheorem*{claim*}{Claim}

\theoremstyle{remark}
\newtheorem{remark}[theorem]{Remark}

\theoremstyle{definition}
\newtheorem{definition}[theorem]{Definition}

\def\calg{\mathcal{G}}

\def\calf{\mathcal{F}}

\def\G{\Gamma}

\def\stab{{\rm{Stab}}}

\def\aut{{\rm{Aut}}}
\def\out{{\rm{Out}}}

\def\gl{{\rm{GL}}}

\def\cd{{\rm{cd}}}

\def\CV{{\rm{CV}}}

\def\<{\langle}
\def\>{\rangle}

\def\pso{{\rm{PSO}}}

\newcommand{\st}{\mathrm{st}}
\newcommand{\lk}{\mathrm{lk}}

\definecolor{olive}{rgb}{0,0.5,0}

\title[Virtual duality and automorphism groups of RAAGs]{A note on virtual duality and automorphism groups of right-angled Artin groups} 
\author{Richard D. Wade \\ \lowercase{\textit{with an appendix by}} Benjamin Br\"uck}

\begin{document}

\begin{abstract}
A theorem of Brady and Meier states that a right-angled Artin group is a duality group if and only if the flag complex of the defining graph is Cohen--Macaulay. We use this to give an example of a RAAG with the property that its outer automorphism group is not a virtual duality group. This gives a partial answer to a question of Vogtmann. In an appendix, Br\"uck describes how he used a computer-assisted search to find further examples. \end{abstract}

\maketitle

\section{Introduction}

The definition of a \emph{duality group} was introduced by Bieri and Eckmann in \cite{BE} in order to describe groups that have a (possibly twisted) pairing between homology and cohomology.  A group $G$ is a \emph{virtual duality group} if some (equivalently, any) finite-index torsion-free subgroup of $G$ is a duality group. By Poincar\'e duality, fundamental groups of closed aspherical manifolds are duality groups. Furthermore, mapping class groups \cite{Harer}, $\gl_n(\mathbb{Z})$ \cite{BorelSerre}, and $\out(F_n)$ \cite{BesF,BSV} are also virtual duality groups for more subtle reasons. As outer automorphism groups of right-angled Artin groups interpolate between $\gl_n(\mathbb{Z})$ and $\out(F_n)$, it is not unreasonable to guess that $\out(A_\G)$ might also be a virtual duality group. The purpose of this note is to show that in general, this is not the case.


%

\begin{thmA}
Let $\G$ be the graph given in Figure~\ref{fig}. Then $\out(A_\G)$ is not a virtual duality group.
\end{thmA}

\begin{figure}[ht]
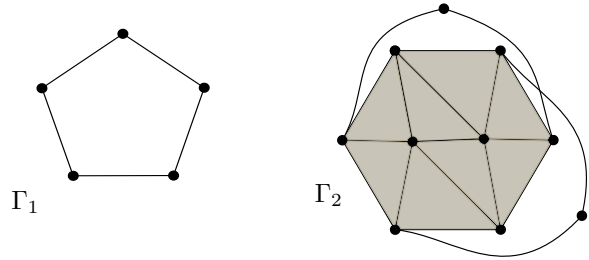
\caption{A graph $\G=\G_1 \sqcup \G_2$ such that $\out(A_\G)$ is not a virtual duality group. The grey triangles are added to show the flag complex $\hat{\G}$ determined by $\G$.}
\label{fig}
\end{figure}

With hindsight, the failure of duality here should not be too surprising, simply because right-angled Artin groups themselves are rarely duality groups. A wonderful theorem of Brady and Meier \cite{BM} shows that a right-angled Artin group $A_\G$ is a duality group if and only if the flag complex $\hat{\Gamma}$ of the defining graph is \emph{Cohen--Macaulay} (see Definition~\ref{d:cm}). To briefly sketch how their result implies Theorem~A, let $\Delta$ be the join of the subgraphs  $\G_1$ and $\G_2$ given in Figure~\ref{fig}. The \emph{Join Lemma}, given below, gives a method for constructing outer automorphism groups of RAAGs with finite-index RAAG subgroups. Its proof follows reasonably quickly from existing results (see Section~\ref{s:join-lemma} for details).

\begin{lemma}[Join Lemma] \label{l:join}
Let $A_{\G_1}$ and $A_{\G_2}$ be two noncyclic right-angled Artin groups with finite outer automorphism groups. If $\G=\G_1 \sqcup \G_2$ is the disjoint union of the two graphs, and $\Delta=\G_1 \star \G_2$ is their join, then $\out(A_\G)$ has a finite-index subgroup isomorphic to the right-angled Artin group $A_{\G_1} \oplus A_{\G_2} \cong A_\Delta$.
\end{lemma}

The construction of the graph $\G$ in Figure~\ref{fig} ensures that both $A_{\G_1}$ and $A_{\G_2}$ have finite outer automorphism groups, so that $\out(A_\G)$ has a finite-index subgroup isomorphic to $A_{\G_1} \oplus A_{\G_2} \cong A_\Delta$. However, the flag complex $\hat{\Delta}$ is not Cohen--Macaulay, so neither $A_\Delta$ nor $\out(A_\G)$ is a (virtual) duality group. Full details are given in Section~\ref{s:proof}. More generally, the Join Lemma gives a way of constructing many examples of RAAGs $A_\Delta$ appearing as finite-index subgroups of $\out(A_\G)$ for certain $\G$.

 Theorem~A gives a very partial answer to Question~3 from Vogtmann's \emph{Groups St. Andrews} lecture notes on automorphism groups of RAAGs \cite{Vogtmann}. We discuss this in more detail and highlight related questions in Section 3. After circulating a draft of this paper, Br\"uck used a computer search to find further examples of RAAGs satisfying the conclusion of Theorem~A. He kindly agreed to describe these examples and his approach in an appendix. More recently, Wiedmer  \cite{Wiedmer2022}  built on these ideas to prove the remarkable result that \emph{every} right-angled Artin group $A_\G$ is commensurable with $\out(A_\Delta)$ for some other RAAG $A_\Delta$. This gives a vast range of examples of RAAGs whose outer automorphism groups are not duality groups. 
 
\paragraph*{\textbf{Acknowledgments.}} We thank the referee for helpful comments and Corey Bregman for feedback on an earlier draft of this paper, particularly Remark~\ref{r:Corey}. Wade is funded by the Royal Society through a University Research Fellowship.

\section{Proof of Theorem A} \label{s:proof}

In this section we provide background, definitions, and expand on the sketch proof given in the introduction to give a full proof of Theorem~A.

\subsection{Right-angled Artin groups}

Let $\G$ be a finite graph with vertex set $V(\G)$. The \emph{right-angled Artin group determined by the graph $\G$} is the finitely presented group $A_\G$ with the presentation:
\[ A_\G= \langle v \in V(\G) \, | \,  vw=wv \text{ if $v$ and $w$ span an edge in $\G$.} \rangle \]
We will always assume subgraphs $\Delta \subset \Gamma$ are \emph{full}, so that two vertices in $\Delta$ are connected by an edge if and only if they are connected by an edge in $\G$. We will also assume that our graphs are \emph{simple}, so that there are no loops and no double edges in $\G$ (this is safe to do as relations given by loops or double edges do not change the group obtained from the above presentation). We make use of the following facts:
\begin{itemize}
\item The centre of a right-angled Artin group $A_\G$ is generated by the vertices $v$ that are adjacent to every other vertex.
\item A RAAG $A_\G$ is one-ended if and only if $\G$ is connected and contains at least two vertices (\cite{GH} proves something much stronger than this - probably the simplest way to show this directly is by using \cite[Lemma~1.1]{GH}).
\end{itemize}

For introductions to RAAGs and their automorphisms, we recommend the survey papers of Charney \cite{Charney} and Vogtmann \cite{Vogtmann}.


\subsection{Flag complexes, the Cohen--Macaulay condition, and duality for RAAGs.}

For a simple graph $\G$, we use $\hat{\G}$ to denote the \emph{flag complex determined by $\G$}.   One can define $\hat{\G}$ as being obtained from $\G$ by filling in any `visible' simplices, or as the the largest simplicial complex on the vertex set $V(\G)$ with the same edge set as $\G$. The \emph{star} of a simplex $\sigma$ is the subcomplex spanned by simplices containing $\sigma$, and the link of $\sigma$ is the subcomplex consisting of simplices $\tau \in \st(\sigma)$ with $\sigma \cap \tau = \emptyset$.

\begin{definition}[Cohen--Macaulay complexes] \label{d:cm}
A finite simplicial complex $X$ is \emph{Cohen--Macaulay of dimension  $n$} if 

\begin{itemize}
\item the reduced homology $\overline{H}_*(X ;\mathbb{Z} )$ is free abelian (possibly trivial) and concentrated in degree $n$, 
\item each maximal simplex is $n$-dimensional, and
\item for each non-maximal $k$-simplex $\sigma$, the reduced homology $\overline{H}_*(\lk(\sigma);\mathbb{Z})$ is free abelian (possibly trivial) and concentrated in degree $n-k-1$.  
\end{itemize}
\end{definition}

A group $G$ is a \emph{duality group of dimension $n$} if there exists a $G$--module $D$ and an element $e \in H_n(G;D)$ such that the cap product with $e$ induces an isomorphism \[H^{n-k}(G; M) \cong H_k(G; D \otimes M) \] for all $k$ and all $G$--modules $M$. If we can take $D=\mathbb{Z}$ in the above then $G$ is a \emph{Poincar\'e duality group} (Bieri and Eckmann allow a nontrivial action on $\mathbb{Z}$ in this definition). We do not work with the definition in this note, instead relying on the following theorem of Brady and Meier.

\begin{theorem} [\cite{BM}, Theorem~C]
Let $\G$ be a finite simple graph. The right-angled Artin group $A_\G$ is a duality group if and only if the flag complex $\hat{\G}$ is Cohen--Macaulay.
\end{theorem}

A group $G$ is a \emph{virtual duality group} if some finite-index subgroup $H$ of $G$ is a duality group. In this case, $H$ is torsion-free, and every finite-index torsion-free subgroup of $G$ is also a duality group. This is well-known and follows directly from results in \cite{BE} but as it is important in what follows we record it below.

\begin{lemma}\label{l:fi}
If $G$ is a virtual duality group and $H$ is a finite-index, torsion-free subgroup of $G$, then $H$ is a duality group.
\end{lemma}

\begin{proof}
As $G$ is a virtual duality group, there exists a finite-index subgroup $H_0$ of $G$ that is a duality group. Let $H'=H\cap H_0$. As $H'$ is finite-index in $H_0$, it is also a duality group by \cite[Theorem~3.2]{BE}. By \cite[Theorem~3.3]{BE}, any torsion-free, finite-index overgroup of a duality group is also a duality group. As $H'$ is also finite-index in $H$, it follows that $H$ is a duality group.
\end{proof}

\subsection{Finiteness conditions for $\out(A_\G)$ and the Join Lemma} \label{s:join-lemma}

In this section, we show that for our example graph in Figure~\ref{fig} the group $\out(A_\G)$ has a finite index subgroup isomorphic to $A_{\Gamma_1} \oplus A_{\Gamma_2}$. For experts, this is the subgroup of $\out(A_\G)$ generated by partial conjugations, and it is finite index as $\G$ is chosen in a way so that $\out(A_\G)$ contains no transvections. We break this down into two steps, starting with conditions that describe when $\out(A_\G)$ is finite:

\begin{proposition}[\cite{CF}, Section~6]\label{p:finite}
Let $\G$ be a finite graph. The group $\out(A_\G)$ is finite if and only if for each vertex $u$:
\begin{itemize}
\item any two vertices in $\G - \st(u)$ are connected by a path in $\G - \st(u)$, and 
\item if $\lk(u) \subset \st(v)$ for some vertex $v$ then $u=v$.
\end{itemize}
\end{proposition}

Going back in the other direction, finiteness of the outer automorphism group $\out(A_\G)$ imposes the following restrictions on the graph $\G$ and its associated RAAG.

\begin{lemma}\label{l:graph_stuff}
Let $\G$ be a graph such that $\out(A_\G)$ is finite. Then
\begin{itemize}
\item $\G$ is connected, so that $A_\G$ is either cyclic or one-ended.
\item If $\G$ is not a single point (so that $A_\G$ is noncyclic), the centre $Z(A_\G)$ of $A_\G$ is trivial.
\end{itemize}
\end{lemma}

\begin{proof}
Let $\G$ be a graph such that $\out(A_\G)$ is finite. The first condition of Proposition~\ref{p:finite} implies that $\G$ has at most two connected components, and if there are exactly two connected components then the star of each vertex $u$ is equal to its own component. Suppose there are exactly two components $C_1$ and $C_2$ such that for each $u \in C_i$ we have $\st(u)=C_i$. This contradicts the second bullet point from Proposition~\ref{p:finite}: either we can find two distinct vertices $u$ and $v$ in the same component that satisfy $\lk(u) \subset \st(v)$, or there is an isolated vertex $u$ whose link is empty and therefore contained in the star of every other vertex. Hence $\G$ is connected.  Furthermore, if there are at least two vertices in $\G$ then $Z(A_{\G})$ must be trivial,  otherwise the star of some vertex is the whole graph and we would have a contradiction to the second bullet point from Proposition~\ref{p:finite}.
\end{proof}

The second proposition we use is a bit more general, and follows from Guirardel and Levitt's work on automorphism groups of free products \cite{GL2}.

\begin{proposition}\label{p:product}
Let $A$ and $B$ be one-ended groups with centres denoted $Z(A)$ and $Z(B)$, respectively. If $A$ and $B$ have finite outer automorphism groups, then $\out(A \ast B)$ has a finite-index subgroup isomorphic to \[ A/Z(A) \oplus B/Z(B). \]
\end{proposition}

\begin{proof}[Sketch proof]
As both $A$ and $B$ are one-ended, $G=A\ast B$ is the \emph{Grushko decomposition} of $G$. In this case, the \emph{Outer space of the free product} (see \cite{GL2}) reduces to a single point: the Bass--Serre tree given by the splitting $A\ast B$ is invariant under the whole of $\out(G)$. By looking at the stabilizer of this tree (\cite[Section 5]{GL2} or alternatively \cite{BJ, Levitt05}) one obtains a subgroup $\out^0(G)$ of $\out(G)$ of index at most two that splits as the following short exact sequence:
\[ 1 \to A/Z(A) \oplus B/Z(B) \to \out^0(G) \to \out(A) \oplus \out(B) \to 1. \]
When both $\out(A)$ and $\out(B)$ are finite, the kernel of this exact sequence is finite-index in $\out(G)$.
\end{proof}

Combining the above allows us to prove the Join Lemma from the introduction:

\begin{proof}[Proof of the Join Lemma]
Let $\G=\G_1 \sqcup \G_2$ be the disjoint union of two graphs with the property that their associated RAAGs $A_{\G_i}$ are noncyclic and have finite outer automorphism groups.  By Lemma~\ref{l:graph_stuff}, both $A_{\G_1}$ and $A_{\G_2}$ are one-ended and have trivial centres. As $A_\G \cong A_{\G_1} \ast A_{\G_2}$, Proposition~\ref{p:product} tells us that  $\out(A_\G)$ has a finite index subgroup isomorphic to $A_{\G_1} \oplus A_{\G_2}$.
\end{proof}

Applying the Join Lemma to our specific example, we have:

\begin{lemma}\label{lemma:out}
Let $\out(A_\G)$ be the graph given in Figure~\ref{fig}. Then $\out(A_\G)$ has a finite index subgroup isomorphic to $A_{\G_1} \oplus A_{\G_2}$.
\end{lemma}

\begin{proof}[Sketch proof]
In order to apply the Join Lemma we check the conditions of Proposition~\ref{p:finite} vertex-by-vertex: that is, for each $u \in \G_i$, the star of $u$ does not separate $\G_i$ and if $\lk(u) \subset \st(v)$ then $u=v$.  For $\G_1$ this is straightforward. For $\G_2$ this is a little harder, but we feel that a line-by-line proof is not beneficial to the paper or the reader. To convince oneself that this holds, we recommend that one looks at the following cases for a vertex $u \in \Gamma_2$: $u$ is a vertex on one of the two strings attached to the hexagon, $u$ is the endpoint of a string, $u$ is one of the two vertices in the middle of the hexagon, and lastly $u$ is one of the two points on the boundary of the hexagon that are not endpoints of a string. These cases cover all vertices in $\G_2$, and in each case the star of $u$ does not separate the graph and the link of $u$ is not contained in the star of any other vertex. 
\end{proof}

\subsection{The proof of Theorem A and the Aut case}

\begin{proof}[Proof of Theorem A]
Let $\G$ be the graph from Figure~\ref{fig}. By Lemma~\ref{lemma:out}, the group $\out(A_\G)$ has a finite-index subgroup isomorphic to $A_{\G_1} \oplus A_{\G_2}$. This is the right-angled Artin group on the graph $\Delta = \G_1 \star \G_2$ formed by taking the join of the two graphs $\G_1$ and $\G_2$. A one-dimensional maximal simplex in $\widehat{\Gamma_2}$ (i.e. an edge on one of the strings) is contained in simplices of dimension at most three in the join $\widehat{\Delta} \cong \widehat{\Gamma_1} \star \widehat{\Gamma_2}$, whereas $\widehat{\Delta} $ is 4-dimensional. As the maximal simplices of $\widehat{\Delta}$ are not of uniform dimension, the flag complex $\widehat{\Delta}$ is not Cohen--Macaulay. Therefore $A_\Delta$ is not a duality group. As \emph{every} finite-index, torsion-free subgroup of a virtual duality group is a duality group (Lemma~\ref{l:fi}), the group $\out(A_\G)$ is not a virtual duality group.
\end{proof}

For completeness, we note that we can obtain a similar result in the Aut case:

\begin{proposition}
If $\G_2$ is the graph given in Figure~\ref{fig}, then $\aut(A_{\Gamma_2})$ is not a virtual duality group.
\end{proposition}

\begin{proof}
As $\out(A_{\Gamma_2})$ is finite and the centre of $A_{\Gamma_2}$ is trivial, the group of inner automorphisms is finite-index in  $\aut(A_{\Gamma_2})$ and is isomorphic to $A_{\Gamma_2}$. As $\widehat{\Gamma_2}$ is not Cohen--Macaulay (the maximal simplices of $\widehat{\Gamma_2}$ do not all have the same dimension), the group $A_{\Gamma_2}$ is not a duality group, so that $\aut(A_{\Gamma_2})$ is not a virtual duality group.
\end{proof}

\section{Further discussion}

In this section we collect some related questions. Most of these problems have appeared elsewhere previously. Vogtmann gave five questions centred around $\mathcal{O}_\G$ (the RAAG version of outer space) at the end of \cite{Vogtmann}. This is question three:

\begin{question}[Vogtmann, \cite{Vogtmann}]
Is $\out(A_\G)$ a virtual duality group? Is there a bordification of $\mathcal{O}_\G$ which is a hybrid
of the Borel--Serre bordification of the symmetric space $\mathbb{D}_n$ and the Bestvina--Feighn
bordification of Outer space $\CV_n$? If so, is bordified $\mathcal{O}_\Gamma$ highly connected at infinity?
\end{question}

The space $\mathcal{O}_\Gamma$ was recently shown to be contractible in work of Bregman, Charney, and Vogtmann \cite{BCV}, and admits a proper action of $\out(A_\G)$. Our example shows that the classification of when $\out(A_\G)$ is a virtual duality group is a delicate problem. Note that even when duality fails, the behaviour of $\mathcal{O}_\G$  at infinity (often described via a bordification) is a very interesting problem. The outer automorphism group constructed in Theorem~A suggests that as well as expecting pieces of a potential bordification to behave like  the Borel--Serre bordification of symmetric space \cite{BorelSerre} and the Bestvina--Feighn bordification of Outer space \cite{BesF,BSV}, we should expect (bordifications of?) Salvetti complexes to also appear, at least in the geometry if not in the actual construction (the explicit construction of $\mathcal{O}_\Gamma$ is in terms of blow-ups and collapses of certain cubulations of $A_\G$: when $\out(A_\Gamma)$ is virtually $A_\Delta$, we should expect $\mathcal{O}_\Gamma$ to be related to the Salvetti complex of $A_\Delta$, but we do not think this relationship is clear from the definitions).

Recently Br\"uck \cite{Brueck} has constructed an $\out(A_\G)$ complex $X_\Gamma$ which is a hybrid of the free factor complex and the Tits building for $\gl_n(\mathbb{Z})$, although $X_\Gamma$ has larger $\out(A_\G)$-stabilizers than one might hope for (for instance, $X_\Gamma$ is trivial for the graph from our main theorem). Br\"uck showed that these complexes are Cohen--Macaulay, so one hope is that $X_\Gamma$ could be used to reduce problems about duality to the behaviour of $\out(A_\G)$-stabilizers in $X_\G$.

\subsection{Obstructions to duality: Fouxe-Rabinovitch groups}

Recall that if \[ \calg=G_1 \ast G_2 \ast \cdots \ast G_k \ast F_n \] is a (not necessarily maximal) free factor decomposition of a group, the associated \emph{Fouxe-Rabinovitch} group is the subgroup of $\out(G)$ consisting of outer automorphisms $\Phi$ that have representatives $\phi_1, \ldots, \phi_k \in \Phi$ such that each representative $\phi_i$ acts as the identity when restricted to $G_i$. This is written as $\out(G; \calg^t)$.

The \emph{decomposition series} constructed in our work with Day \cite{DW2} break up $\out(A_\G)$ into consecutive quotients that are either free-abelian, $\gl(n,\mathbb{Z})$, or certain Fouxe-Rabinovitch groups. If all the consecutive quotients are virtual duality groups, then so is $\out(A_\G)$ (when working with more general subnormal series one has to be a little bit careful about the passage to finite index subgroups, however here we can use congruence subgroups). This leads us to ask the following question, which also appeared briefly in \cite{DSW}.

\begin{question} \label{q:fr}
Let $\mathcal{G}$ be a free factor decomposition of a RAAG $A_\G$. When is $\out(A_\G; \calg^t)$ a virtual duality group?
\end{question}

For the graph in Figure~\ref{fig}, the RAAG subgroup $A_\Delta$ of $\out(A_\G)$ appears as the Fouxe-Rabinovitch group $\out(A_\G;\{A_{\G_1},A_{\G_2}\}^t)$. Here the failure of duality comes from failure of the factor groups to be duality groups. 

However, we conjecture that there is another possible obstruction. In comparison with known examples in the literature, we expect natural classifying spaces for duality groups to look uniformly of the same dimension as the group (i.e. all maximal simplices/cells are of dimension $d=\cd(G)$). For Fouxe-Rabinovitch groups, classifying spaces of minimal dimension can be obtained as a blow-up of the spine of relative Outer space by replacing each simplex $\sigma$ with a copy of $\sigma \times \textrm{E}\stab(\sigma)$ \cite{DSW} (at least after passing to an appropriate torsion-free f.i. subgroup). However, as we will see below, one can have simplices in relative Outer space of the same dimension whose stabilizers have different (geometric/cohomological) dimensions. As the spine is uniform, the resulting blow-up will not be uniform - some maximal simplices will be of dimension strictly less than $\cd(G)$.

\begin{figure}[ht]
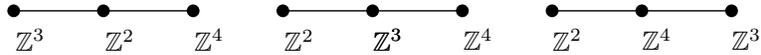
\caption{Three points in the spine of the relative outer space for $A_\G=\mathbb{Z}^2 \ast \mathbb{Z}^3 \ast \mathbb{Z}^4$, whose $\out(A_\G)$-stabilizers are isomorphic to $\mathbb{Z}^2$, $\mathbb{Z}^3$, and $\mathbb{Z}^4$, respectively.}
\label{fig:2}
\end{figure}

The group of \emph{pure symmetric outer automorphisms} (also called basis-conjugating automorphisms) $\pso(A_\G)$ is the subgroup of $\out(A_\G)$ given by all outer automorphisms $\Phi$ whose representatives $\phi\in \Phi$ send every generator $v\in V(\G)$ to a conjugate of itself. For $A_\G=\mathbb{Z}^2 \ast \mathbb{Z}^3 \ast \mathbb{Z}^4$, the group $\pso(A_\G)$ forms a Fouxe-Rabinovitch group. Either using the work in \cite{DW} or by working directly from a presentation, one can show that
\[\pso(\mathbb{Z}^2 \ast \mathbb{Z}^3 \ast \mathbb{Z}^4) \cong \mathbb{Z}^2 \ast \mathbb{Z}^3 \ast \mathbb{Z}^4. \]

This gives another example of a Fouxe-Rabinovitch group that is not a virtual duality group. In this case, $\out(A_\G)$ fits in a short exact sequence \[ 1 \to  \mathbb{Z}^2 \ast \mathbb{Z}^3 \ast \mathbb{Z}^4 \to \out(A_\G) \to \gl_2(\mathbb{Z}) \oplus \gl_3(\mathbb{Z}) \oplus \gl_4(\mathbb{Z})\to 1, \]  so unlike our first example the group $\pso(A_\G)$ is not finite-index in $\out(A_\G)$. It would be interesting to know how this decomposition is reflected in the geometry of the outer space for the RAAG.

The following simplification of Question~\ref{q:fr} is interesting in its own right:

\begin{question}
Let $\calf$ be a free factor system in a free group $F_N$. When is $\out(F_N; \calf^t)$ a virtual duality group?
\end{question}

This is true when $\mathcal{F} = \emptyset$ (by Bestvina and Feighn \cite{BesF}), and in the case where $\mathcal{F}=\mathbb{Z} \ast \mathbb{Z} \ast \cdots \ast \mathbb{Z}$ is the free factor decomposition that determines the pure symmetric automorphism group (by Brady, McCammond, Meier, and Miller \cite{MR1872129}). In both of the above cases, simplex stabilizers behave well in the spine of the (relative) outer space, and the problem illustrated in Figure~\ref{fig:2} does not occur. However, this is not true for an arbitrary free factor system of $F_N$.

\subsection{Commensurability problems}

Given the construction in Theorem~A, it seems worthwhile to repeat the following question, a version of which appeared  as Question~1.1 in \cite{DW}.
\begin{question}\label{q:comm}
When does $\out(A_\G)$ have a finite-index subgroup isomorphic to a right-angled Artin group $A_\Delta$? Conversely, which RAAGs appear as such finite-index subgroups?
\end{question}

One can also ask similar questions up to quasi-isometry. The above problem is discussed at some length in the introduction of \cite{DW}, so we will limit ourselves to mentioning more recent developments. Notably, the work of Aramayona and Martinez--Perez \cite{MR3461054} on when $\out(A_\G)$ can have property (T) has been recently extended by Sale \cite{Sale}. Through this work, as well as Guirardel and Sale's work on \emph{vastness properties} and $\out(A_\G)$ \cite{GS}, we now have much better control over the behaviour of outer automorphism groups of RAAGs that, roughly speaking, do not look like $\out(F_n)$ or $\gl(n,\mathbb{Z})$. These results give reasons to be more optimistic about the tractability of the first part of Question~\ref{q:comm}. The second part of this question seems much harder, given the fact that quasi-isometry and commensurability classification problems for RAAGs themselves are incredibly difficult (see \cite{MR3842063, Margolis}). However, the Join Lemma does provide a way to construct families of examples $A_\Delta$ that are finite index in $\out(A_\G)$ for some $\G$ (and now the work of Wiedmer greatly extends this \cite{Wiedmer2022}). It is also worth noting that \cite[Question~1.2]{DW} gave a more general recognition problem about RAAGs, which was later answered in the negative by Bridson \cite{MR4085052}.

\begin{remark}\label{r:Corey}
Let $\G$ be the example graph in Figure~\ref{fig} and $A_\Delta=A_{\Gamma_1} \times A_{\Gamma_2}$ be the associated finite-index RAAG subgroup of $\out(A_\G)$. Corey Bregman pointed out some extensions to our main example where $\out(A_\G)$ behaves similarly but $\G$ is connected. If $\G'$ is the cone of our example graph with an additional vertex, then $\out(A_{\G'})$ is commensurable with $A_\Delta \times \mathbb{Z}^{|\G|} $ (there is an additional free abelian group generated by transvections by the additional central element and these commute with the existing partial conjugations). Rather than taking the cone, one can take $\G''$ to be the join of $\G$ with two vertices (while working with flag complexes, we can think of this as the suspension of $\G$), in which case $\out(A_{\G''})\cong \out(A_\G) \times \out(F_2)$, so is commensurable with $A_\Delta \times F_2$, as $\out(F_2)$ is virtually free. Further connected examples are given in the appendix.
\end{remark}

\appendix

\section{Computer-assisted construction of further examples, by Benjamin Br\"uck}
In \cite{DW}, Day--Wade give sufficient and necessary conditions for when the group of pure symmetric outer automorphisms $\pso(A_\Gamma)$ is itself a RAAG. This gives another way to find examples of RAAGs whose outer automorphism groups are not  virtual duality groups. In particular, computer calculations that used the conditions of \cite{DW} revealed the two examples depicted in Figure~\ref{fig:9vertex_ex}; both are connected and have only 9 vertices.

\begin{figure}
\begin{center}
\includegraphics[scale=1]{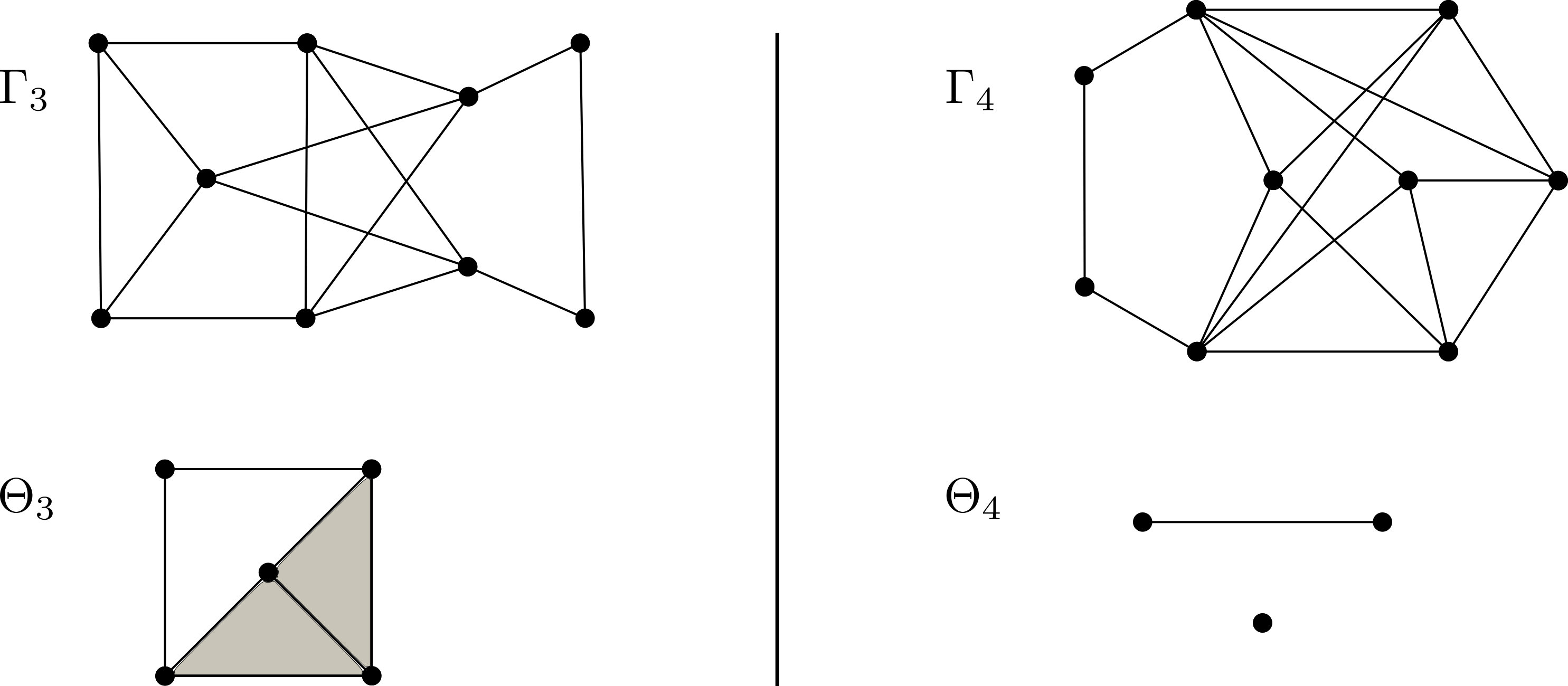}
\end{center}
\caption{Two graphs $\Gamma_i$ with 9 vertices such that $\out(A_{\G_i})$ is not a virtual duality group. The top row shows the defining graphs $\G_i$, the bottom row shows graphs $\Theta_i$ such that $\pso(A_{\G_i})\cong A_{\Theta_i}$.}
\label{fig:9vertex_ex}
\end{figure}

To obtain these examples, one proceeds as follows: If $\out(A_\G)$ contains no transvections (which is equivalent to the second condition of Proposition~\ref{p:finite}), then $\pso(A_\G)$ has finite index in $\out(A_\G)$.
Day--Wade define for a vertex $a\in \G$ a \emph{support graph} $\Delta_a$ that encodes the structure of connected components of $\Gamma-\st(a)$, see \cite[Definition 1.3]{DW}. They show that $\pso(A_\G)$ is isomorphic to a RAAG if and only if for all $a$ this support graph is a forest. Furthermore, if this is the case, they explain how to obtain a graph $\Theta$ such that $\pso(A_\G)\cong A_\Theta$. 
If the flag complex $\hat{\Theta}$ associated to this graph is not Cohen--Macaulay, then $A_\Theta$ is not a duality group. Hence $\out(A_\G)$ cannot be a virtual duality group.

In order to use these arguments for finding explicit examples, we wrote a python script that generates random graphs using the Erd\H{o}s--R\'enyi model with varying numbers of vertices and edge probabilities and follows the steps in the previous paragraph. The script does not actually verify all the conditions for Cohen--Macaulayness but instead just checks whether the corresponding flag complexes are pure, i.e.~whether all the maximal simplices have the same dimension.
A summary of the procedure can be found in Algorithm~\ref{algorithm_duality}.
All of these conditions can easily be checked using simple operations on graphs provided by the python package \texttt{networkx} \cite{SciPyProceedings_11}.
The code is available at \url{https://github.com/benjaminbrueck/computations_for_roars/blob/main/duality_Out(RAAGs).ipynb}.

\begin{algorithm}
\caption{Finding $\G$ such that $\out(\G)$ is not a virtual duality group}
\label{algorithm_duality}
\begin{algorithmic}
\While{not found\_example}
	\State{generate a random graph $\G$}
	\If{$\out(\G)$ has no transvections}
		\State{\textit{(The group $\pso(A_\G)$ is finite index in $\out(A_\G)$.)}}
		\If{every support graph is a forest}
			\State{\textit{(The group $\pso(A_\G)$ is isomorphic to a RAAG $A_\Theta$.)}}
			\State{calculate $\Theta$}
			\If{there are maximal cliques of different size in $\Theta$}
				\State{\textit{(The flag complex $\hat{\Theta}$ is not Cohen--Macaulay, so $A_\Theta$ is not a duality }} 
				\State{\, \textit{group and thus $\out(A_\G)$ is not a virtual duality group.)}}
				\State{found\_example = True}
				\State{\Return $\G$}
			\EndIf
		\EndIf
	\EndIf
\EndWhile
\end{algorithmic}
\end{algorithm}

The two examples in Figure~\ref{fig:9vertex_ex} were obtained using this method. It is not hard to verify by hand that all the support graphs are forests (in fact, none contains more than one edge) and to compute the associated graphs $\Theta_i$. 

These graphs have 9 vertices and 15 and 17 edges, respectively. We believe that they are the examples with the minimal number of vertices that can be obtained using this procedure. 
Computer calculations show that the obstruction above does not appear among the 1,253 simple graphs with at most 7 vertices (as provided by the Atlas of Graphs \cite{RWatlasgraphs}). There are 12,346 graphs on 8 vertices and 274,668 graphs on 9 vertices \cite{OEIS}. For these, we did not have a list available. However, the computer generated $10^6$ random graphs with 8 vertices without finding an example. For 9 vertices, around $10^7$ random graphs were generated and up to isomorphism, the examples presented in Figure~\ref{fig:9vertex_ex} were the only ones that appeared.

Of course, one could also try to use other obstructions in order to find examples where the flag complex $\hat{\Theta}$ is not Cohen--Macaulay. In addition to checking whether $\hat{\Theta}$ is pure, we also looked for disconnected graphs of dimension at least one. However, this did not lead to new findings with 9 or fewer vertices. We doubt that looking for further obstructions to Cohen--Macaulayness would be very helpful as calculations showed that for such small graphs, the dimension of $\hat{\Theta}$ is usually low.

\bibliography{vdbib}

\begin{thebibliography}{10}

\bibitem{OEIS}
{OEIS Foundation Inc. (2021), The On-Line Encyclopedia of Integer Sequences}.
\newblock \url{https://oeis.org/A000088}.

\bibitem{MR3461054}
J.~Aramayona and C.~Mart\'{\i}nez-P\'{e}rez.
\newblock On the first cohomology of automorphism groups of graph groups.
\newblock {\em J. Algebra}, 452:17--41, 2016.

\bibitem{BJ}
H.~Bass and R.~Jiang.
\newblock Automorphism groups of tree actions and of graphs of groups.
\newblock {\em J. Pure Appl. Algebra}, 112(2):109--155, 1996.

\bibitem{BesF}
M.~Bestvina and M.~Feighn.
\newblock The topology at infinity of {${\rm Out}(F_n)$}.
\newblock {\em Invent. Math.}, 140(3):651--692, 2000.

\bibitem{BE}
R.~Bieri and B.~Eckmann.
\newblock Groups with homological duality generalizing {P}oincar\'{e} duality.
\newblock {\em Invent. Math.}, 20:103--124, 1973.

\bibitem{BorelSerre}
A.~Borel and J.-P. Serre.
\newblock Corners and arithmetic groups.
\newblock {\em Comment. Math. Helv.}, 48:436--491, 1973.
\newblock Avec un appendice: Arrondissement des vari\'et\'es \`a coins, par A.
  Douady et L. H\'erault.

\bibitem{MR1872129}
N.~Brady, J.~McCammond, J.~Meier, and A.~Miller.
\newblock The pure symmetric automorphisms of a free group form a duality
  group.
\newblock {\em J. Algebra}, 246(2):881--896, 2001.

\bibitem{BM}
N.~Brady and J.~Meier.
\newblock Connectivity at infinity for right angled {A}rtin groups.
\newblock {\em Trans. Amer. Math. Soc.}, 353(1):117--132, 2001.

\bibitem{BCV}
C.~Bregman, R.~Charney, and K.~Vogtmann.
\newblock Outer space for {RAAG}s.
\newblock {\em arXiv:2007.09725}, 2020.

\bibitem{MR4085052}
M.~R. Bridson.
\newblock On the recognition of right-angled {A}rtin groups.
\newblock {\em Glasg. Math. J.}, 62(2):473--475, 2020.

\bibitem{Brueck}
B.~Br{\"u}ck.
\newblock Between buildings and free factor complexes: {A} {C}ohen--{M}acaulay
  complex for {O}ut({RAAG}s).
\newblock {\em Journal of the London Mathematical Society}, 105:251--307, Jan.
  2022.

\bibitem{BSV}
K.-U. Bux, P.~Smillie, and K.~Vogtmann.
\newblock On the bordification of outer space.
\newblock {\em J. Lond. Math. Soc. (2)}, 98(1):12--34, 2018.

\bibitem{Charney}
R.~Charney.
\newblock An introduction to right-angled {A}rtin groups.
\newblock {\em Geom. Dedicata}, 125:141--158, 2007.

\bibitem{CF}
R.~Charney and M.~Farber.
\newblock Random groups arising as graph products.
\newblock {\em Algebr. Geom. Topol.}, 12(2):979--995, 2012.

\bibitem{DSW}
M.~B. Day, A.~W. Sale, and R.~D. Wade.
\newblock Calculating the virtual cohomological dimension of the automorphism
  group of a {RAAG}.
\newblock {\em Bull. Lond. Math. Soc.}, 2020.

\bibitem{DW}
M.~B. Day and R.~D. Wade.
\newblock Subspace arrangements, {BNS} invariants, and pure symmetric outer
  automorphisms of right-angled {A}rtin groups.
\newblock {\em Groups Geom. Dyn.}, 12(1):173--206, 2018.

\bibitem{DW2}
M.~B. Day and R.~D. Wade.
\newblock Relative automorphism groups of right-angled {A}rtin groups.
\newblock {\em J. Topol.}, 12(3):759--798, 2019.

\bibitem{GH}
D.~Groves and M.~Hull.
\newblock Abelian splittings of right-angled {A}rtin groups.
\newblock In {\em Hyperbolic geometry and geometric group theory}, volume~73 of
  {\em Adv. Stud. Pure Math.}, pages 159--165. Math. Soc. Japan, Tokyo, 2017.

\bibitem{GL2}
V.~Guirardel and G.~Levitt.
\newblock The outer space of a free product.
\newblock {\em Proc. Lond. Math. Soc. (3)}, 94(3):695--714, 2007.

\bibitem{GS}
V.~Guirardel and A.~Sale.
\newblock Vastness properties of automorphism groups of {RAAG}s.
\newblock {\em J. Topol.}, 11(1):30--64, 2018.

\bibitem{SciPyProceedings_11}
A.~A. Hagberg, D.~A. Schult, and P.~J. Swart.
\newblock Exploring network structure, dynamics, and function using networkx.
\newblock In G.~Varoquaux, T.~Vaught, and J.~Millman, editors, {\em Proceedings
  of the 7th Python in Science Conference}, pages 11 -- 15, Pasadena, CA USA,
  2008.

\bibitem{Harer}
J.~L. Harer.
\newblock The virtual cohomological dimension of the mapping class group of an
  orientable surface.
\newblock {\em Invent. Math.}, 84(1):157--176, 1986.

\bibitem{MR3842063}
J.~Huang.
\newblock Commensurability of groups quasi-isometric to {RAAG}s.
\newblock {\em Invent. Math.}, 213(3):1179--1247, 2018.

\bibitem{Levitt05}
G.~Levitt.
\newblock Automorphisms of hyperbolic groups and graphs of groups.
\newblock {\em Geom. Dedicata}, 114:49--70, 2005.

\bibitem{Margolis}
A.~Margolis.
\newblock Quasi-isometry classification of right-angled {A}rtin groups that
  split over cyclic subgroups.
\newblock {\em Groups, Geometry, and Dynamics}, 14(4):1351--1417, 2020.

\bibitem{RWatlasgraphs}
R.~C. Read and R.~J. Wilson.
\newblock {\em An atlas of graphs}.
\newblock Oxford Science Publications. The Clarendon Press, Oxford University
  Press, New York, 1998.

\bibitem{Sale}
A.~Sale.
\newblock On virtual indicability and property {(T)} for outer automorphism
  groups of {RAAG}s.
\newblock {\em arXiv:2011.03576}, 2020.

\bibitem{Vogtmann}
K.~Vogtmann.
\newblock {$GL(n,\Bbb Z)$}, {$Out(F_n)$} and everything in between:
  automorphism groups of {RAAG}s.
\newblock In {\em Groups {S}t {A}ndrews 2013}, volume 422 of {\em London Math.
  Soc. Lecture Note Ser.}, pages 105--127. Cambridge Univ. Press, Cambridge,
  2015.

\bibitem{Wiedmer2022}
M.~Wiedmer.
\newblock Right-angled artin groups as finite-index subgroups of their outer
  automorphism groups.
\newblock {\em arXiv:2209.02033}, 2022.

\end{thebibliography}
\bibliographystyle{abbrv}

\begin{flushleft}
Richard D.\ Wade\\
Mathematical Institute, University of Oxford\\
Oxford, UK. OX2 6GG\\
\emph{e-mail: }\texttt{wade@maths.ox.ac.uk} 
\end{flushleft}
~
\begin{flushleft}
Benjamin Br\"uck \\
ETH Zurich \\
Department of Mathematics \\
R\"amistrasse 101 \\
8092 Zurich, Switzerland \\
\emph{e-mail: }\texttt{benjamin.brueck@math.ethz.ch}
\end{flushleft}

\end{document}